\documentclass[12pt,reqno]{amsart}
\usepackage{amsmath, amssymb, esint, cite}
\usepackage[margin=1in]{geometry}
\usepackage{amsmath}
\usepackage{amsthm}
\usepackage{fullpage,amsbsy,url,enumerate,comment,colonequals}
\usepackage{parskip}

\usepackage{bbm}

\usepackage[pdftex]{graphicx,color}
\pdfoutput=1

\newtheorem{theorem}{Theorem}[section]
\newtheorem{prop}[theorem]{Proposition}
\newtheorem{lemma}[theorem]{Lemma}

\newtheorem{conjecture}[theorem]{Conjecture}

\newtheorem{question}[theorem]{Question}

\def\R{\mathbb{R}}

\def\G{\Gamma}

\def\O{\Omega}

\def\a{\alpha}
\def\b{\beta}

\def\de{\partial}

\def\f{\varphi}

\newcommand{\eps}{\epsilon}

\title{The Two Hyperplane Conjecture}
\date{}
\author{David Jerison}
\address{David Jerison, Department of Mathematics,  Room 2-272,
Massachusetts Institute of Technology, Cambridge, MA 02139, USA}
\email{jerison@math.mit.edu}

\thanks{Supported in part by NSF Grant DMS 1500771, a Simons Fellowship, and Simons Foundation grant 
(601948, DJ)}

\numberwithin{equation}{section}

\begin{document}

\begin{abstract}
We  introduce a conjecture that we call the {\it Two Hyperplane
Conjecture}, saying that an isoperimetric surface that divides
a convex body in half by volume is trapped between 
parallel hyperplanes.   The conjecture is motivated by an
approach we propose to the {\it Hots Spots Conjecture} of J. Rauch
using deformation and Lipschitz bounds for 
level sets of eigenfunctions.  We will relate this approach
to quantitative connectivity properties of level sets of 
solutions to elliptic variational problems, including 
isoperimetric inequalities, Poincar\'e inequalities, Harnack inequalities, 
and NTA (non-tangentially 
accessibility).  This paper mostly asks questions rather than answering them,
while recasting known results in a new light. 
Its  main theme is that the level 
sets of least energy solutions to scalar variational problems
should be  as simple as possible.

\end{abstract}

\maketitle

\section{Introduction}

This paper is dedicated to Carlos Kenig.   Every one of this author's collaborations 
with Carlos was a joy.   Early in our careers we collaborated on a paper \cite{JK}  
in which we introduced a quantitative,  scale-invariant form of connectivity of
domains which we called the NTA (non-tangentially accessible) property.   
We describe here how several such quantitative 
scale-invariant connectivity properties are related and how they 
can be used as a  tool in nonlinear analysis.  
Although in our original paper, the NTA property was
a hypothesis used to prove a boundary Harnack estimate for
a linear operator, already in 1986,  Aguilera, Caffarelli and 
Spruck \cite{ACS} proved an NTA property for free boundaries. In 2011, De Silva
and Jerison \cite{DSJ} showed that the boundary Harnack estimate
can serve as a stepping stone towards full regularity of free boundaries.     
In \cite{DSJ}, we explained the parallel roles played by the boundary Harnack estimate 
for free boundaries and a  Harnack inequality on area minimizing  surfaces
due to Bombieri and Giusti \cite{BG}.   The present paper elaborates on that
analogy.

The theme of this paper is that level sets of solutions to 
linear and semilinear elliptic equations
and singular limits of these surfaces, which are area minimizing
surfaces, isoperimetric surfaces, and free boundaries, should be as
simple and as regular as possible, even in high dimensions.   We will illustrate
what we mean starting with a version of the hot spots conjecture.
We will then
broaden the discussion to include other problems in the calculus of variations.
This will lead us to raise many more questions related 
to quantitative forms of connectivity.  We expect that positive answers to
these questions will ultimately overcome the barrier represented by the Simons 
cone in dimension 8 to conjectures of De Giorgi type in the
case of isoperimetric hypersurfaces in convex domains
as opposed to area-minimizing surfaces.  
The initial steps concerning connectivity
have already been taken by Bombieri and Giusti and in work using
stability by Sternberg and Zumbrun \cite{StZ1,StZ2,StZ3}, as well as 
the work of  by Rosales et al. \cite{RCBM}.   Our questions are also motivated by high dimensional convex geometry.

The paper is organized as follows.  In Section 2 and 3 we state a version
of the hot spots conjecture in all dimensions and propose an avenue to resolving
it using Lipschitz bounds on level sets and deformation.  In Section 4 we describe 
how a Harnack inequality can be used to prove related Lipschitz bounds,
not quite the ones we want.   In Section
5 we describe how Harnack inequalities are related to other forms of
quantitative connectivity.  In Section 6, we describe the large scale,
global estimates that are the most fundamental reflection of connectivity, including
what we are calling the two hyperplane conjecture.   Very roughly speaking,
these conjectures are an attempt to quantify the extent to which isoperimetric
subsets of convex sets are convex.    We expect that some version
of such bounds is what is needed to complete our program to prove
Lipschitz and higher regularity of level sets in all dimensions.  Finally, in Section 7 we give a short
list of known results that represent modest qualitative evidence in favor
of some of our conjectures.   

We thank Guy David for discussions over many years and
permission to mention our ongoing work.  We thank Emanuel Milman
for the proof, included here, that our two hyperplane conjecture implies the 
KLS Hyperplane Conjecture.   We also thank Ronen Eldan, 
Larry Guth, Bo'az Klartag, and Frank Morgan for helpful conversations.

\section{The Hot Spots Conjecture and Lipschitz Level Sets}

As initially formulated, the {\it Hot Spots Conjecture} of Jeff Rauch \cite{R}
says that, generically, 
the hottest point of a perfectly insulated domain
tends to the boundary as time tends to infinity. This poetic description is a disguise for 
a somewhat more precise formulation that for a generic bounded domain in Euclidean
space, the maximum (and minimum) of a least energy non-constant Neumann eigenfunction is  achieved on the boundary.   

The hot spots conjecture is quite far from being resolved.  
Burdzy and Werner showed 
that the conjecture is false for certain multiply-connected planar regions, and
Burdzy constructed a 1-connected counterexample in \cite{Bur}.   Some partial results
can be found in work by Kawohl, Ba\~nuelos-Burdzy, Jerison-Nadirashvili, Atar-Burdzy, Miyamoto, and Siudeja.
 A proof for acute triangles was announced only a few months ago by Judge and Mondal \cite{JM} where references to this earlier work can be found.  
Despite the resistance of this conjecture, we believe it for simply-connected planar domains and for convex domains in all dimensions.

Our preferred version of the hot spots conjecture concerns centrally symmetric
convex bodies.  We denote by $-\O$ the set $\{-x: x\in \O\}$, and we say that
a convex set $\O$ is symmetric if $-\O = \O$.   

\begin{conjecture} \label{conj:hotspots}  (version of the hot spots conjecture of
J. Rauch) \  Suppose  $\O\subset \R^n$ is a bounded, open, convex, 
and $-\O = \O$.  If $u$ is a Neumann eigenfunction with the least
non-zero eigenvalue, then there is a direction $a\in \R^n$, $a \neq 0$, 
in which $u$ is strictly monotone:
\[
a \cdot \nabla u(x) > 0 \quad \mbox{for all} \quad x\in \overline{\O} \quad \mbox{such that} \quad
|a\cdot x| < \max_{y\in \overline\O} a\cdot y.
\]
Hence, $u$ achieves its maximum and minimum on $\de \O$
on some subset of the points where $a\cdot x$ achieves its maximum and mininum,
respectively.
\end{conjecture}
This conjecture says that the level sets are as simple
as possible.  For example, if the domain is smooth
and strictly convex, then there are exactly two critical
points in $\overline \O$.  The direction $a$ must be the direction from the unique
minimum to the unique maximum, and, taking $a$ as the vertical
direction, every intermediate level set is a smooth graph.  In particular, each interior level set has the
trivial topology of an embedded $(n-1)$-disk bounded by an $(n-2)$-sphere in the boundary.

The author and N. Nadirashvili proved  Conjecture \ref{conj:hotspots} 
for convex planar domains with two axes of symmetry in \cite{JN} by the method of continuity.   
Atar and Burdzy \cite{AB} proved monotonicity of the lowest energy
nonconstant Neumann eigenfunction
in planar domains bounded by two Lipschitz graphs with Lipschitz constant
less than or equal to $1$ using probabilistic arguments.  They call such
domains ``lip" domains.   

Without the symmetry $-\O = \O$, the eigenfunctions can fail
to be monotone.  For example,
in the case of a narrow circular sector, the extrema occur at the vertex
and  on the  
entire arc of the circle, and one sees easily that there are convex perturbations
of the circular arc for which the critical point structure on the boundary is
complicated. 
Indeed, this complexity is already present in acute triangles
and is a reason  why the recent proof of the hot spots
conjecture in that case is subtle.  For many acute triangles, the least energy
non-constant Neumann eigenfunction is not strictly monotone in any direction
and has four
critical points on the boundary, namely, local maxima and minima at the vertices
and a saddle point on one side (see \cite{JM}).    
The rectangle is a borderline symmetric case showing
that the failure of convexity can give eigenfunctions that are not monotone
and whose level sets are more complicated. 
In a rectangle, the direction $a$ points parallel to the longer sides, and the
maximum and minimum are achieved on the entire  length of the shorter sides.   Because
the maximum and minimum are achieved on whole segments,
small perturbations of a rectangle can produce 
a complicated critical point structure of the eigenfunction.

We propose to use the method of continuity  to
prove Conjecture \ref{conj:hotspots}. 
The idea is to start from a symmetric convex body such as a ball, 
with explicit, simple level sets and deform it to any symmetric convex body 
while preserving the graph property.   I am indebted to Nikolai Nadirashvili
for teaching me about deformation as a method in the context
of level sets of eigenfunctions.
The particular implementation of the method proposed here requires a quantitative 
estimate, namely that
the level sets are Lipschitz graphs.  
In our setting the rounded ends
of convex bodies make it so that as the level set gets closer and
closer to the maximum or minimum, the Lipschitz constant will typically
tend to infinity.     For that reason, we formulate 
a quantitative conjecture for the nodal set only.
\begin{conjecture} \label{conj:nodal} \  There is a constant $C_n$ such
that if $\O$ is a bounded, open, convex, symmetric
subset of $\R^n$, and $u$ is a Neumann eigenfunction with the least
non-zero eigenvalue, then there is a Lipschitz function $\f:\R^{n-1}\to \R$ such
that in  suitable coordinates, 
\[
\{x\in \O: u(x) = 0\} = \{(x',x_n)\in \O: \R^{n-1}\times \R:  x_n = \f(x')\}, \quad |\nabla \f| \le c_n.
\]
\end{conjecture}
The hope is that this kind of Lipschitz bound is preserved under deformation. 
We believe that the deformation approach can also be used
to give a different proof of the Atar-Burdzy monotonicity theorem
for lip domains. (Their theorem applies to all level
sets, not just the nodal set.)  For convex domains, we expect that 
a different approach is required for level sets near the extrema.
The level sets near zero should have a uniform Lipschitz bound,
but near the maximum and minimum the slope of the level sets gets
larger and larger, and the property that we can hope is stable
is convexity.  In other words, one should
try to prove that the property super level sets 
$\{x\in \O: u(x) \ge c\}$ are convex  for $c$ sufficiently close to $\max u$
is stable under deformation, and similarly near the minimum.

In this paper we do not offer a proof or even a complete pathway to
proving either of these conjectures.   Instead, we suggest that to 
succeed with the approach by deformation, we need to address even more 
fundamental questions that are relevant to a large variety of problems in the calculus of variations.   Although the  eigenfunction solves a linear elliptic equation 
$\Delta u = -\lambda u$,  
its level sets depend in a decidedly nonlinear way on $u$.   Given
that linearity does not appear to be a central feature of the problem,
we will consider minimizers $u$ of functionals of the form
\[
\int_\O [|\nabla f|^2 + W(f) ] \, dx
\]
for which the Euler-Lagrange equation $2\Delta u = W'(u)$ is semilinear.
The common feature of these semilinear equations is a deformation
maximum principle (see Theorem 2.2 \cite{CS} and Proposition \ref{prop:weakmaxprin} below).   

In the case of a double well potential $W(u) = (1-u^2)^2$ the equation
$\Delta u = 2(u^3- u)$ is known as the Allen-Cahn equation.   As
shown by Modica and Mortola \cite{MM}, a scaled, singular limit of the functional 
is the area functional.   This leads us to the isoperimetric problem.  

A Borel set $E$ is called isoperimetric in $\O$ if it has 
least perimeter in $\O$ among  all Borel subsets of $\O$
with the same volume fraction $|E|/|\O|  = \a$ for some $0 < \a < 1$.
The (relative) perimeter of $E$ in $\O$, $P(E,\O)$, is defined as 
\[
P(E,\O) = \liminf_{\eps \to 0} \frac{|E_\eps \cap \O|-|E|}{\eps} \,  
\]
with $E_\eps$ the $\eps$-neighborhood of $E$ and $|\cdot |$ the Lebesgue measure.
Isoperimetric sets are only defined up
to a set of Lebesgue measure $0$.  But there is always an open representative
(see \cite{G}).  We will always use this representative, so that 
the topological boundary of $E$ is meaningful.  By the well
known regularity theory of de Giorgi \cite{G}, 
the perimeter of $E$ in $\O$ coincides with
the $(n-1)$-dimensional Hausdorff measure
of the topological boundary of $E$ 
minus the part of the boundary that coincides with $\de \O$, that is,
\[
P(E,\O) = H_{n-1}(\O \cap \de E).
\]
We emphasize that $\O$ is open, so that the portion $(\de \O) \cap (\de E)$ is
not counted in this relative perimeter.

 In \cite{Mod}, Modica proved that 
if $u_\eps$ minimizes the constrained problem 
\[
\int_\O [\eps |\nabla f|^2 + (1-f^2)^2] \, dx, \quad \int_\O f \, dx   = 0,
\]
then there is a subsequences of $u_\eps$ converging in $L^1(\O)$ as $\eps \to 0$
to a function $u_0$,
and there are open sets $E_\pm\subset \O$, $\O = E_+\cup E_-$,  $|E_\pm| = |\O|/2$,
such that $u_0 = 1_{E_+} - 1_{E_-}$ almost everywhere, and
the interface $\de E_+ \cap \de E_-$ has least $H_{n-1}$ measure
among  all partitions of $\O$ into sets $E_\pm$ of equal volume.
Thus, both $E_\pm$  are isoperimetric with $\a = 1/2$. (We  can obtain any
 volume fraction $\a$ by changing the constraint on the integral of $f$ over
 $\O$ to another constant.)

To make the analogy with eigenfunctions more explicit, recall that the isoperimetric inequality
\[
\left(\int_\O |f(x)|^{n/(n-1)} \, dx\right)^{(n-1)/n} 
\le C_* \int_\O |\nabla f| \, dx, \quad \int_\O f(x) \, dx = 0,
\]
is valid with best constant $C_* =  |\O|^{(n-1)/n} /2 H_{n-1}(\de E_+ \cap \de E_-)$
and one extremal is $u_0(x) = 1_{E_+}- 1_{E_-}$ (interpreting the right
side integral by duality, in the sense used for functions of bounded variation; see \cite{G}).  
The Euler-Lagrange equation for this constrained problem is that the interface has a constant mean curvature and meets the boundary of $\O$ perpendicularly. 
Likewise, extremals for the Poincar\'e inequality
\[
\int_\O |f(x)|^2 \, dx \le (1/\lambda) \int_\O |\nabla f|^2 \, dx, \quad \int_\O f(x) \, dx = 0,
\]
are Neumann eigenfunctions of $\O$ with the smallest nonzero eigenvalue  $\lambda$,
and the level sets of these eigenfunctions meet the boundary of $\O$ perpendicularly.

In analogy with Conjecture \ref{conj:nodal} for nodal sets, we have the following
conjecture for isoperimetric sets. 
\begin{conjecture} \label{conj:Lip}
If $\O$ is a bounded, convex domain in $\R^n$,
then the boundary $\O\cap \de E$ of every open isoperimetric subset 
$E$ of $\O$ 
is a Lipschitz graph.
\end{conjecture}
Notice that in the isoperimetric version,  
we have dropped the symmetry assumption on $\O$ and
we are making this conjecture for all volume fractions, not just $\a = 1/2$. 
For small volume fraction, isoperimetric subsets are close to
the intersection of $\O$ with a ball centered on the boundary.  Thus the
behavior for small isoperimetric sets is 
simpler and better than the behavior we see for level sets of eigenfunctions
near the maximum and minimum (see the discussion above for triangles, for instance).

By standard regularity theory of constant
mean curvature surfaces, Conjecture \ref{conj:Lip}  implies that the isoperimetric
hypersurface $\O\cap \partial E$ is smooth.  Thus, if the conjecture is
true, there is a sharp contrast between local regularity of solutions
of the isoperimetric problem, a Neumann problem with a single scalar
constraint,
and the Dirichlet type problem in which one prescribes boundary values
for area minimizing surfaces. The first
non-smooth area minimizing surface that was discovered is the Simons cone, 
\[
S = \{x\in \R^{2n}: x_1^2 + \cdots + x_n^2 = x_{n+1}^2+ \cdots + x_{2n}^2\}, \quad 2n \ge 8.
\]
In dimension $8$ and higher, 
$S$  has the least area in the unit ball $B$ centered at the origin,
among all surfaces in $B$ with the same boundary $S\cap \partial B$.
It also happens that $S$ divides the convex set $B$ in half by volume, but
$S$ is not an isoperimetric set.  The 
set with least area that divides the ball in half by volume 
is a bisecting hyperplane. We argue here that the complexity
of the Simons cone is dictated by its boundary $ S\cap \de B = S^{n-1} \times S^{n-1}$.
In the isoperimetric problem this boundary is allowed to move, and
$S$ is not stable for any $n$, as  demonstrated by Sternberg and Zumbrun 
(see the last section of this paper).

To test our conjectures for large $n$,  it's natural,
in the spirit of high dimensional convex geometry, 
 to generalize them to the case of log-concave densities.
Denote by $\mu = w(x) dx$ a probability measure on $\R^n$ with log-concave density 
$w$,
that is, $V:\R^n \to \R$ is convex, and 
\[
w(x) = e^{-V(x)}, \quad \int_{\R^n} w(x) \, dx  = 1.
\]
We say that $\mu$ is strictly log-concave if $V$ is strictly convex. 
The perimeter of an Borel set $E$ relative to $\mu$ is defined
as 
\[
P_\mu(E) : = \, \liminf_{\eps \to 0+} \, \frac{\mu(E_\eps) - \mu(E)}{\eps} \, .
\]
An isoperimetric set $E\subset \R^n$ is an Borel set such that 
\[
\mu(E) = \a,  \quad P_\mu(E) \le P_\mu(F)
\]
for every Borel set $F$ satisfying the volume constraint $\mu(F) = \a$.
As before, $E$ is equivalent up to a set of Lebesgue measure zero
to an open set, so we may restrict our attention to open sets $E$.  The open representative is sufficiently regular that 
\[
P_\mu(E) = \int_{\de E} w \, dH_{n-1}
\]
If we allow $V$ to take the value $+\infty$,
open, bounded convex domains can be viewed as the special
case $w(x) = 1_\O$ with $V(x) = 0$ for $x\in \O$ and
$V(x) = \infty$ for $x\notin\O$.

\begin{conjecture} \label{conj:logconcave.Lip}  There is a constant $C= C(\a,n)$
such that for every strictly log-concave probability measure $\mu = e^{-V}dx$  
on $\R^n$
and every open isoperimetric subset $E$ with $\mu(E) = \a$ such that
after suitable rotation, $\de E$ is a Lipschitz graph:
\[
\de E = \{(x',x_n)\in \R^{n-1}\times \R:  x_n = \f(x')\}, \quad |\nabla \f(x')| \le C(\a,n),
\]
at least for some suitable growth of $V$ that includes
the case of bounded convex domains $1_\O dx$ in the limit.
\end{conjecture}
Some extra condition besides strict  log concavity on $V$ is necessary.  
The reason is that the 
conjecture rules out bounded isoperimetric
sets, but numerical examples in $\R^2$ of Kolesnikov and Zhdanov \cite{KZ} show
that for $V$ of the form $c+ |x|^\b$, $1 \le  \b < 2$, there are bounded isoperimetric sets
at some volume fractions.  On the other hand, isoperimetric sets
are never relatively compact in bounded domains $\O$.  Indeed,
if the isoperimetric set $E$ were relatively compact in $\O$, it
would be a ball.    After translation, we can assume
the ball is tangent to the boundary of $\O$.  Then we can perturb
$E$ using a homothety to find a competitor with smaller perimeter:
For sufficiently small $t>0$ and suitable $x$,  there is a set $(tx + (1+t)E)\cap \O$
with the same volume as $E$ but a smaller perimeter relative
to $\O$ because it shares a portion of its boundary with $\de \O$. 
This is a contradiction.

Note that the choice of vertical direction $x_n$ can depend on $E$.     
For example, if $w(x) = c_n e^{-|x|^2}$, the gaussian, C. Borell's theorem \cite{Bor}
says that the isoperimetric sets are all half spaces, allowing for every orientation. Thus,
in this case, the conjecture is true and the vertical direction must be chosen depending 
on the choice of $E$.   

Any weight that is a product of log-concave one-dimensional weights will
have isoperimetric sets that are half spaces (see \cite{Bob}).  
If the one-dimensional
weight is strictly log-concave (that is, $V$ is strictly convex) then the
all isoperimetric sets are half spaces.  On the other hand, without
strictness, there can be other extremals.   As observed in Rosales et al \cite{RCBM},
for the weight $w(x) = e^{-|x|}/2$ for $x\in \R$, every finite interval of volume $\a$ 
containing the origin is an extremal if $0 <\a \le 1/2$.  Taking the product with
another log-concave function in a second variable, we have an isoperimetric
set whose boundary is two parallel lines.

\section{The method of continuity}

A well known theorem of Bombieri, De Giorgi and Miranda \cite{BDGM}
gives an {\em a priori} Lipschitz bound on minimal surfaces that are graphs.
\begin{theorem} \label{thm:gradbound} Let $\f \in C^\infty(B_1)$ be a solution to the minimal surface (mean curvature zero) equation
\[
\mbox{\rm{div}} \left(\frac{\nabla \f}{\sqrt{1+|\nabla \f|^2}}  \right) = 0,
\]
with $|\f| \le M$.  Then there is a constant $C = C(n,M)$ such that 
\begin{equation} \label{eq:gradbound}
|\nabla \f|\le C \quad \mbox{on} \ B_{1/2}.
\end{equation}
\end{theorem}

Motivated by this result, we formulate a tentative plan to prove
Conjecture \ref{conj:logconcave.Lip} by deformation,
with several gaps in the argument highlighted.  
Consider an isoperimetic set $E_1$ relative to a probability
measure $\mu_1= w_1(x) dx$.  

{\em Step 1.} \ Find 
a continuous family of log-concave probability measures $\mu_t = w_t(x) dx$, 
$0 \le t \le 1$, with $w_t(x)>0$ for  all $x\in \R^n$ and a {\em continuous} family
$E_t$ of isoperimetric sets with $\mu_t(E_t)   = \a$,
with respect to $t$ in  Hausdorff distance when intersected with
compact subsets, 
such that, for instance,  
\[
E_0 = \{(x',x_n): x_n \ge 0\}.
\]

{\em Step 2.}  Show that for all $t$,
$0\le t \le1$, $\de E_t$ {has an appropriate $L^\infty$ bound}
replacing the bound by $M$ in Theorem \ref{thm:gradbound}.
If the mean curvature were zero, scaling dictates 
linear growth in the sense that
\begin{equation} \label{eq:Mbound}
\de E_t \cap \{(x',x_n): |x'|\le K\}  \subset \{(x',x_n)\in \R^{n-1}\times \R: |x'| \le K, \   |x_n| \le  MK\},
\end{equation}
for sufficiently large $K$.

Let $A \subset [0,1]$ be the set of $t\in [0,1]$ such that 
\[
\de E_t   =  \{(x',x_n)\in \R^{n-1}\times \R:   x_n  = \psi_t(x')\},
\]
with $\psi_t $ with Lipschitz bound $2C$.   We show
that the set $A= [0,1]$ by showing that it is closed,
relatively open in $[0,1]$ and nonempty ($0\in A$). 
Standard
regularity theory of prescribed mean curvature surfaces implies
that $\psi_t$ is smooth for $t\in A$.  

{\em Step 3.} \ {\em Prove a gradient bound theorem like the one of Bombieri et al} that applies to stable solutions to the isoperimetric problem as opposed
to area minimizers, showing that given a bound like \eqref{eq:Mbound}
one has gradient bounds on the graph.
Then for all $t\in A$, we have the uniform Lipschitz bound $|\nabla \psi_t| \le C$.
Lipschitz bounds are preserved under convergence of sets $E_t$ under
Hausdorff distance, so the set $A$ is closed.  Standard pertubation theory
shows that for any $t\in A$, for every $s$ near $t$,
the function $\psi_s$ is smooth, and in particular a graph.  Hence,
$s\in A$ and $A$ is open.

This outline has substantial gaps.  First we have not proved
that we can deform $E_1$ continuously to a known $E_0$ and $w_0$.
In our view, the hardest step, by far, is Step 2.  At present it is not
even formulated accurately for constant mean curvature as opposed to
zero mean curvature.  In addition to being revised to account
for nonzero mean curvature,  the cone condition \eqref{eq:Mbound} 
will need to be scaled
appropriately at the boundary of convex domains.  Likewise,
the variant of \eqref{eq:Mbound} to be used in Step 3 
in the case of log concave densities should also be 
rescaled appropriately at infinity.
Despite the
fact that we are not providing a precise hypothesis for
Step 3, we will be able to explain the mechanism of proof, which
has many precedents and is in the spirit of regularity theorems of the form ``flat implies Lipschitz".  We will begin this discussion in the next section.

In the section following, we will discuss the type of bound we expect is
needed to establish estimates like the one in Step 2, which represent
a global, large scale starting place for regularity at smaller scales provided
by Step 3.

\section{Gradient bounds}

Here we will outline the proof of the gradient bound Theorem \ref{thm:gradbound}
given in our work with De Silva \cite{DSJ}. This is far from the first and
far from the only proof of the theorem, but it reveals how connectivity
is the central mechanism.  Moreover, the method should work for
a wide range of variational problems.  For instance in 
 \cite{DSJ}, an analogous proof is given that applies to free boundaries.

In brief, there is a Harnack inequality that tells us that the points
of $S$ are connected to each other in a quantitative way.
The Harnack inequality can be used to show
that the points of $S$ ``help each other".  If $S$ has
bounded slope on a set of large measure, then it has bounded slope
at all points.  

In more detail, consider  the function $\f$ of Theorem \ref{thm:gradbound} and
define the surface
\[
S = \{(x,\f(x)): x\in \R^{n-1}, \ |x| < 1\}.
\]
The idea is to translate $S$ up vertically by $\eps$, that is,
to consider $S^\eps = S +(0,\eps)$.
Estimating the Lipschitz constant
of $\f$ is equivalent to proving a quantitative separation
between $S$ and $S^\eps$, namely that for sufficiently small $\eps>0$,
the distance in $\R^n$ from $S^\eps$ to $S$ is bounded below by
$c\eps$ for some $c>0$.   We achieve this by showing that for an appropriate
function $w>0$ the deformation of $S$ 
in the normal direction:
\[
S_t  : = \{P +tw(P)\nu(P) : P \in  S\},
\quad 
P : = (x,\f(x)), \quad \nu(P) = \frac{(-\nabla \f(x),1)}{\sqrt{1 + |\nabla \f(x)|^2}} \, .
\]
does not touch $S^\eps$ for $0 \le t\le c\eps$, when restricted to a sufficiently
small ball $B_\beta$, with $\beta>0$ to be chosen later.

Assume without loss of generality, $\f(0) = 0$.
We will construct a function $w$ on $S$
so that $w(P) \ge c>0$ on all of $S\cap B_\beta$ and so
that for all $0 \le t \le c\eps$, the surface $S_t$ does
not touch $S^\eps$.  There will be two reasons why no
touching is allowed.  On a ``good" set of
large measure, we will have $\nabla
\f$ is bounded, using the $L^\infty$ bound on
$\f$ and the mean curvature zero equation for $\f$.  On the remaining
portion of $S\cap B_\beta$ we define $w$ by solving a boundary value problem 
on $S$ and invoke an intrinsic Harnack inequality  of
Bombieri and Giusti.

For $t$ sufficiently small, since
$S$ has mean curvature zero ($H(P,S) =0$), the 
mean curvature of $S_t$ is given by 
\begin{equation}\label{eq:Hvariation}
H(P_t,S_t) = t(\Delta_S  w(P) + |A_S|^2 w(P)) + O(t^2),
\end{equation}
with $P_t: = P + tw(P)\nu(P)$, $\Delta_S$ the Laplace-Beltrami operator
on $S$, and $|A_S|$ the Hilbert-Schmidt norm of the second fundamental 
form $A_S$ of $S$.

Using a test function, it's not hard to check that for some
dimensional constant $C$, 
\[
\int_{B_{1/2}} |\nabla \f| \, dx \le C M^2. 
\]
It follows from Chebyshev's inequality that there
is a ``good" set $G\subset S\cap B_{\beta}$ whose
measure is a significant fraction of the whole,
on which the slope of $\f$ is bounded and we can set $w=1$.  
For the complement, we solve the equation
\[
\Delta_S w = 0 \quad \mbox{in} \ S \setminus G;
\quad w = 1 \ \mbox{on} \ G; \quad w = 0 \ \mbox{on}\  \de S.
\]
In particular, $w$ is a non-negative supersolution $\Delta_S  w \le 0$ 
in $S$.

The intrinsic Harnack inequality for supersolutions
of $\Delta_S$ due to Bombieri and Giusti \cite{BG} is
stated as follows.
\begin{theorem} \label{thm:HarnackBG} (Bombieri-Giusti)  Let $S$ be
an area minimizing surface in the ball $B_1$ in $\R^n$.
There are dimensional
constants $\beta>0$ and $C<\infty$ such that if $w\ge 0$
and $\Delta_S w \le 0$ on $S$, then 
\[
\int_{B_\beta\cap S} w \, d\sigma \le C \inf_{B_\beta \cap S} w
\]
\end{theorem}

By Theorem \ref{thm:HarnackBG}, $w\ge c>0$ on 
all of $S \cap B_\beta$.   We can perturb $w$ slightly so that
we have strict inequality $\Delta_S w >0$ on $S\setminus G$
by relaxing the inequality on $S\cap B_\beta$ to $w \ge c/2$. 
In particular,  $\Delta_S w + |A_S|^2 w >0$ on $(S\cap B_\beta)\setminus G$
and hence $H(P_t,S_t) >0$ for $0 \le t \le c\eps$ for sufficiently small $\eps$,
and all  $P\in (S\cap B_\beta)\setminus G$.   Recall that for $P\in G$,
$P_t\notin S^\eps$ for all $0 \le t \le c\eps$.  Let $t^*>0$ be the smallest
time such that $P_{t^*} \in S^\eps$ for some $P\in (S\cap B_\beta) \setminus G$. 
If $t^* \le c\eps$ we will derive a contradiction.  Indeed,
$S_{t^*}$ touches $S^\eps$ from below at $P_{t^*}$, which implies
$H(P_{t^*},S_{t^*}) \le 0$.  On the other hand $H(P_t,S_t) >0$ for all $t\le c\eps$.
In all, we have $S_t$ stricly below $S^\eps$ on $B_\beta$
for all $0 \le t \le c\eps$. 

\section{Quantitative Connectivity}

In this section, we will see that level sets of elliptic equations satisfy
various connectivity properties.  As we have already
seen, the Harnack inequality
of Bombieri and Giusti, a quantitative form
of the property that least area hypersurfaces are connected,
implies
a gradient bound on minimal graphs.   As is well known, 
Nash-De Georgi estimates imply full $C^\infty$ regularity,
starting from this gradient bound.  
Thus, this form of connectivity can be an important intermediate step towards regularity.

The proof of the Harnack inequality given by Bombieri and Giusti 
starts from
an even more basic connectivity property, due to Almgren and De Giorgi.

Given an open set $E\subset \R^n$, we say $\de E$ has least area or minimizes
area if for every ball $B_r$ and every set $F$,
\[
F \setminus B_r =  E\setminus B_r \quad \mbox{implies} 
\quad P(E,B_{2r}) \le P(F,B_{2r}).
\]
(Recall that $P(E,\O)$ is the perimeter of $E$ relative to $\O$.)
If $A$ is a set of locally finite perimeter, then, by definition
the indicator function of $A$, $1_A$, is a locally a BV function
and the distribution $D 1_A$ is a vector-valued measure.  Let $\| D 1_A\|$  
denote its total variation.  We say that $A$ is {\em decomposable} if there exist
$A_k$, $k=1,\ 2$ of locally finite perimeter with nonzero measure, $|A_k| >0$,  
such that
\[
1_A = 1_{A_1} + 1_{A_2}; \quad \|D1_A\| = \|D1_{A_1}\| + \|D1_{A_2}\|.
\]
The measure-theoretic connectivity theorem of Almgren and De Giorgi
(Theorem 1, \cite{BG}) is as follows.

\begin{theorem}\label{thm:connectivity-ADG}  (Almgren and De Giorgi) 
Let $E\subset \R^n$ be such that $\de E$ has least area,
then $E$ is not decomposable. In particular, $E$ an $\de E$ are connected.
\end{theorem}

De Giorgi deduced from this theorem the following scale-invariant
isoperimetric inequality on the least perimeter surface $ \de E$. 

\begin{theorem}\label{thm:isoperimetricDG}  (De Giorgi) 
Let $E\subset \R^n$ be an open subset such that $\de E$ has least area.
Let 
\[
S = B_r \cap \de E
\]
Suppose that  $S_k\subset S$, $k = 1, \ 2$,  are locally finite perimeter sets in $S$ 
such that $1_S = 1_{S_1} + 1_{S_2}$.  Then 
\[
 H^{n-2}(\de_* S_1) = H^{n-2}(\de_* S_2) \ge  c\, \min_k 
 H^{n-1}(S_k \cap B_{\beta r})^{(n-2)/(n-1)},
 \]
 with  $c>0$ and $\beta>0$ dimensional constants. By $\de_*S_k$ we mean the
 reduced (essentially the regular part of) the boundary of $S_k$ relative to $S$.
 \end{theorem}
 
 From this, Bombieri and Giusti deduce an integrated form of
 the isoperimetric inequality 

\begin{theorem}\label{thm:isoperimetricBG}  (Bombieri-Giusti)
Let $E\subset \R^n$
be an open subset such that $\de E$ has least area.  There is a dimensional
constants $C$ and $\beta>0$ such that for every $f\in C^1(B_r)$, 
\[
\min_c \int_{B_{\beta r} \cap S}|f-c|^{(n-1)/(n-2)}\, d\sigma \le C \left(\int_{B_r\cap S} |\nabla_S f| \, d\sigma  \right)^{(n-1)/(n-2)}\, .
\]
(Here $\nabla_S$ is the component of the gradient tangent to $S$.)
\end{theorem}
The rest of the proof by Bombieri and Giusti of Theorem \ref{thm:HarnackBG} 
follows
a parallel structure to Moser's Harnack inequality for linear elliptic operators
with bounded measurable coefficients.  

Sternberg and Zumbrun (see Theorem 2.5 \cite{StZ2}) proved qualitative results analogous to
the ones described above for  stable isoperimetric surfaces in convex 
domains.   Let $E$ be an open isoperimetric subset of an open convex set 
$\O\subset \R^n$.  Let $S = \O \cap \de E$.  Then they computed the condition for 
stability 
\begin{equation} \label{eq:poincare-SZ} 
\int_{S}  f^2 \|A_S \|^2 \, dH^{n-1} + \int_{\bar S \cap \de \O}   f^2 
A_{\de \O}(\nu_S,\nu_S) \, dH^{n-2} 
 \le \int_{S} |\nabla_S f|^2 \, dH^{n-2}
\end{equation}
for all $f$ such that 
\[
\int_S f \, dH^{n-1}  = 0.
\]
They deduced (see Theorem 2.6 \cite{StZ2}) that if $\de \O$ is
strictly convex ($A_{\de\O}(\tau,\tau)>0$  for every nonzero tangent vector 
$\tau$), then $S$ is connected.   
In further work in \cite{StZ1, StZ3},
Sternberg and Zumbrun used analogous ideas to prove that
that level sets of solutions to equations of Allen-Cahn type are connected in strictly convex domains.  

The second variation relative to densities was computed by 
V. Bayle,  and the analogous connectivity statement for isoperimetric sets 
relative to strictly log-concave densities was deduced from
the corresponding stability property by Rosales et al in \cite{RCBM}.

In summary, the stability property (positivity of the second variation)
is a Poincar\'e type inequality related to the square $|\nabla_S f|^2$ 
rather than an isoperimetric inequality related to the first power $|\nabla_S f|$. 
Nevertheless, it should lead to quantitative forms of connectivity analogous
to the Harnack inequality of Bombieri and Giusti.  This should lead
in turn to a proof of regularity by deformation in each of the contexts
mentioned above, in particular for level sets of solutions to elliptic equations,
isoperimetric sets.

Formally, the Allen-Cahn equation $2\Delta u = W'(u)$ on $\R^n$ with
$W(u) = (1-u^2)^2$ can be rescaled by $U_\eps(x) = u(x/\eps)$, yielding 
as $\eps\to 0$  a function $U_0$ satisfying $U_0(x) \equiv 1$ on $U_0>0$, 
$U_0 \equiv -1$ on $U_0 < 0$, and  $2\Delta U_0 = a\delta'(U_0)$,
with $\delta'$ the derivative of the Dirac delta function (as a function
of $U_0$, not $x$) and  for the constant $a$ given by 
\[
a = \int_{-1}^1 W(u) \, du \, .
\]

Next,  consider the functional
\[
\int \, [|\nabla v|^2 + F(v)] \, dx
\]
whose  Euler-Lagrange equation is 
$2\Delta u = F'(u)$  on $\R^n$.  If, in contrast with the Allen-Cahn case,
$F' \in C_0^\infty(\R)$ with
\[
\int_\R F'(v) \, dv = c > 0,
\]
then the appropriate scaling is $v_\eps (x) = \eps v(x/\eps)$
and the rescaled functional tends to the Alt-Caffarelli functional
\[
\int [|\nabla v_0|^2 + c \, 1_{\{v_0>0\} } ] \, dx \, .
\]
After rescaling the extremals by defining $H_\eps(x) = \eps u(x/\eps)$,
the formal limit  as $\eps \to 0$ is $H_0$ satisfying
$2\Delta H_0 = c\delta(H_0)$.   This time the function $H_0$ 
is harmonic, but not necessarily constant in the positive
and negative phases $H_0 > 0$ and $H_0 < 0$.  
The set $S : = \de \{x: H_0(x) >0\}$ is known as the free
boundary.  If $S$ is smooth, then the 
condition across $S$ corresponding to the 
equation $2\Delta H_0 = c\delta(H_0)$ is derived
formally in \cite{CS} by considering
the case $n=1$ and the equation 
\[
(d/dx) (u')^2  = 2u'' u' = c f(u) u'
\]
From this one obtains on $\de \{x: H_0(x) >0\}$ 
\[
|\nabla H_0^+|^2  - |\nabla H_0^-|^2 = c \   \mbox{ on }  \ S. 
\]
In the so-called one-phase case, we have $H_0\ge 0$.
In other words the negative phase is trivial ($H_0\equiv 0$ on $H_0 \le 0$).
The jump condition is then written
\[
|\nabla H_0^+|^2 = c \   \mbox{ on }  \ S. 
\]
The main goal of regularity theory is to show that $S$ is 
as smooth as possible and that the jump (free boundary) condition 
is valid.

In \cite{DS}, De Silva proved that the one phase free boundary graphs
are smooth.  In \cite{DSJ}, we proved the quantitative version
analogous to the theorem of Bombieri, De Giorgi and Miranda
in the minimal surface case.  

\begin{theorem} \label{thm:gradboundDSJ} (De Silva and Jerison \cite{DSJ})
Let $\f\in C^\infty(\R^{n-1})$, $\f(0)=0$, and suppose that 
$S =\{(x,\f(x)): |x|<1\}$ is a one-phase free boundary, that is, 
there is a positive harmonic function $H$ on $\{(x,y): y > \f(x), \ |x|<1\}$ 
such that $H=0$ and $|\nabla H| = 1$ on $S$.  
If $|\f| \le M$ on $|x|<1$, then there is a constant $C$ depending only
on $M$ and $n$ such that 
\[
\max_{|x|\le 1/2} |\nabla \f| \le C.
\]
\end{theorem}

The proof of Theorem \ref{thm:gradboundDSJ}
 starts with a quantitative
connectivity property of the positive phase known as non-tangential accessibility.
Recall that an open set $\O\subset \R^n$ is {\em non-tangentially accessible (NTA)}
if there is a constant $C$ such that the following two properties hold.

1.  (Harnack chain condition) For every $M$, there is $N$ such that for every $r>0$, if
$|x_1-x_2|\le Mr$, $B_r(x_1)\subset \O$, and $B_r(x_2)\subset \O$,
then there are points, $y_k$, $k=0,\dots, N$ such
that 
\[
y_0 = x_1,\ y_N= x_2, \quad B_{r/C}(y_k)\subset \O, \ \mbox{and} \ 
|y_k -y_{k+1}| \le r/C.
\]

2.  (Corkscrew condition)  For every $r>0$ and every $x\in \de \O$, 
there  are points $y_1\in \O$ and $y_2\in \R^n\setminus \O$ such that
\[
B_r(y_1) \subset \O,\quad B_r(y_2)\subset \R^n\setminus \O, \quad 
|y_k- x| \le Cr.
\]

Property 1 is uniform connectivity at  scale $r$ and distance
$r$ from the boundary.  It implies that positive harmonic functions 
at this scale and distance from the boundary satisfy the ordinary
Harnack comparability property.  Property
2 in the interior provides a ball of size comparable to $2^k$ in $\O$ 
that is at distance $2^k$ from $x\in \de \O$ for every $k$.
With Property 1, this implies a kind of
non-tangential accessibility of the boundary resembling
an interior cone condition, but, 
as the name corkscrew was intended to suggest, the path may 
be a twisty one rather than a cone.

In \cite{DS} (see also \cite{ACS} for a version with a constraint)
De Silva proved that the set $\{(x,y): y > \f(x)\}$ of Theorem
\ref{thm:gradboundDSJ} satisfies the NTA property in the range
$|x| < 1/2$ and $r< 1/C$.   Next one deduces a
boundary Harnack inequality, proved in \cite{JK}.
This boundary 
Harnack inequality plays the same role in
the proof as the Harnack inequality for
the Laplace-Beltrami operator on minimal surfaces plays
in the proof of the analogous result for minimal surfaces.  
Here is a statement of the theorem in the global case.

\begin{theorem}\label{thm:boundaryHarnackJK} (Boundary Harnack Principle 
\cite{JK})  Suppose that $\O$ is a non-tangentially accessible
domain in $\R^n$.  There is a constant $C$ depending only on dimension and
the NTA constants above, such that if $x_0\in \de \O$, then
for every $r>0$, if $u$ and $v$ are positive harmonic functions in $B_r\cap \O$
that vanish on $B_r \cap \de \O$, then
\[
\sup_{B_{r/C}(x_0)\cap \O} \frac{u(x)}{v(x)} \le C \inf_{B_{r/C}(x_0)\cap\O} \frac{u(x)}{v(x)}
\]
\end{theorem}
The reader is referred to \cite{DSJ} for the rest of the
proof of Theorem \ref{thm:gradboundDSJ}.  

We hope to have illustrated by these several instances that quantitative 
connectivity can play an important role in regularity theory.  Thus we
will seek a few more versions of it.  

First, note that it's an easy exercise to show that if $\O$ is
an NTA domain in all of $\R^n$, then the distance between points along paths
restricted to $\O$ is comparable to the distance along straight lines.
In the case of area-minimizing surfaces, we can prove two other versions
of quantitative connectivity.  

\begin{theorem} \label{thm:NTAminimal}  
Suppose $E\subset \R^n$ is an open subset such that $\de E$ has least area.
Then $E$ and $\R^n\setminus E$ are NTA domains. In particular,
the intrinsic distance in $E$ is comparable to the straight line distance.
\end{theorem}
This theorem is announced here for the first time, but the details of
the proof will be published elsewhere.  The first main
ingredient is an isoperimetric inequality in dimension $n$ of
the form
\begin{lemma}\label{lem:leastperimeter}
Suppose $E\subset \R^n$ is an open subset such that $\de E$ has least area.
For every pair of sets $E_k$, $k=1, \ 2$ of locally finite perimeter such
that $E = E_1 \cup E_2$, $|E_1\cap E_2| = 0$, we have
\[
P(E_1,E\cap B_r)= P(E_2,E\cap B_r) \ge c\, \min_k |E_k \cap B_{\a r}|^{(n-1)/n}
\]
for some dimensional constants $c>0$ and $\a >0$.
\end{lemma}
This lemma is proved by the same method in dimension $n$ 
as opposed to dimension $n-1$ as the isoperimetric estimate
of De Giorgi (Theorem \ref{thm:isoperimetricDG}).  It seems likely
to have been known to De Giorgi.  This lemma can then be combined
with techniques and results of David and Semmes 
(see (6.18) page 251 of \cite{DS2}; Lemma  6.14 of \cite{DS2} p. 250;
Theorem 1.20, Proposition 1.18 and Defintition 1.16 of \cite{DS1})
to prove  Theorem \ref{thm:NTAminimal}. Note that David and Semmes
already proved a significant part of the result by showing that $E$ is what is known
as a John domain.

\begin{theorem} \label{thm:intrinsicdist} (G. David and D. Jerison)
Let $E\subset \R^n$ be an open subset such that $\de E$ has least area.
Then the intrinsic distance along paths in $\de E$ is comparable
to the extrinsic distance along straight lines in $\R^n$. 
\end{theorem}
This theorem, announced here for the first time,
is due to Guy David and the present author.  The proof will be published elsewhere.
The theorem implies, in particular, that if one
uses balls in the intrinsic distance on $S$, one can take
$\beta = 1$ in Theorems \ref{thm:isoperimetricDG} and \ref{thm:isoperimetricDG}.
Set $S = \de E$.  The main lemma used to prove Theorem \ref{thm:intrinsicdist}
 is that for every $x\in S$,
\[
H^{n-1}(S\cap \tilde B_r(x)) \ge c r^{n-1}.
\]
where $\tilde B_r(x)$ is the ball around $x$ of radius $r$ in the  {\em intrinsic} distance
on $S$. 

Guy David and the author are working on quantitative versions of
the results of Sternberg and Zumbrun that should apply to
isoperimetric sets in convex domains as well as level sets
of stable solutions to semilinear elliptic equations.
We have not yet formulated (much less proved) the appropriate 
Harnack inequalities in the setting of semilinear equations.  
On the other hand, the mechanism for deformation arguments
should be valid in this generality.  The notion of comparison
of level sets of solutions is already present and use extensively
in the work of Caffarelli on free boundary regularity.  The
basis for it is the following deformation maximum principle.

\begin{prop} \label{prop:weakmaxprin}
Suppose that  $U$ is a bounded, open set, $u$ is
a smooth function on $\bar U$, and $v_t$, $0\le t \le 1$ 
is a continuous family of smooth functions on $\bar U$ satisfying
$\Delta u = f(u)$ and $\Delta v_t >f(v_t)$ on $U$. 
If $v_0 < u$ on $\bar U$ and $v_t < u$ on $\de U$ 
for all  $0 \le t \le 1$, then $v_t < u$ on $U$ for all $0 \le t \le 1$. 
\end{prop}

Proof:  If $t>0$ is the smallest value for which $v_t < u$ fails, then
there is $x_0\in U$ such that $v_t(x_0) = u(x_0)$ and by continuity
in $t$, $v_t(x) \le u(x)$ for all $x\in \bar U$.   But then
\[
f(u(x_0)) = f(v_t(x_0)) < \Delta v_t(x_0) \le \Delta u(x_0) = f(u(x_0))
\]
which is a contradiction.

It is important to realize that quantitative connectivity results
like Theorems \ref{thm:isoperimetricDG}, 
\ref{thm:isoperimetricBG}, \ref{thm:NTAminimal},
\ref{thm:intrinsicdist}, and \ref{thm:HarnackBG} apply
in all dimensions to area mininizers like the Simons cone.
To get further regularity, one must have a stronger hypothesis
that, at a minimum, rules out the Simons cone.   In
our opinion, the right place to look is isoperimetric problems
in convex domains, which we discuss in the next section.

\section{The two hyperplane conjecture}

Here we formulate conjectures saying under what circumstances a
cone condition and/or flatness holds and discuss many interrelated conjectures.  
We are also interested in flatness from one side.   A theorem
of Miranda (see Theorem 7 \cite{BG}) says that an area-minimizing
hypersurface in a half space must be a plane.

We begin with a conjecture about the convex hulls  of isoperimetric sets.
\begin{conjecture} \label{conj:hull}   Let $\O\subset \R^n$ be a bounded,
open, convex set.  Let $\mu = 1_\O dx $.  
For every isoperimetric
set $E$ with volume fraction $\a$, $0 < \a < 1$, we have the following.

a) $\rm{hull}(E)$ is a proper subset of $\O$.

b)  $\mu(\rm{hull}(E)) \le (1-c(\a,n))\mu(\O) $  for some $c(\a,n)>0$.

c)  The constant $c(\a,n)$ can be taken independent of the dimension $n$,
at least for $\a<1/2$. 
\end{conjecture}

A log-concave version is as follows.
\begin{conjecture} \label{conj:logconcave.hull}   Let $\mu$ be a
finite, strictly log-concave measure on $\R^n$.  For every isoperimetric
set $E$ with volume fraction $\a$, $0 < \a < 1/2$, there are convex
sets $G_1$ and $G_2$ such that $G_1 \subset E \subset G_2$ and
\[
\mu(G_1) \ge c(\a,n) \mu(\R^n), \quad \mu(G_2) \le (1-c(\a, n)) \mu(\R^n)
\]
for a constant $c(\a,n)>0$.  Moreover, $c(\a,n)$ can be chosen independent
of $n$. 
\end{conjecture}

The idea behind these conjectures is to get started at the largest scale
by finding a big convex set on each side of the isoperimetric interface.  
This is reminiscent of the ball/corkscrew condition in the NTA property
imposed at the most important scale.  Put another way,
we want to quantify the 
extent to which the isoperimetric sets are themselves almost convex.
Convexity is {\em ipso facto} a form of connectivity.

In dimension  $n=1$,
the isoperimetric sets for strictly log-concave measures 
are rays (see \cite{Bob}), so Conjecture \ref{conj:logconcave.hull} is true.
When the density is not strictly log-concave, the conclusion can fail in dimension $1$.   
As observed by Rosales et al \cite{RCBM}, for the weight $w(x)= e^{-|x|}$ on $\R$, 
which is log-concave but not strictly so, the isoperimetric sets are not unique.
Although one of them is a ray, there are other minimizers for 
volume fractions $\a \le 1/2$ are 
the union of two rays.    For such minimizers the hull of $E$ is all of $\R$.

Examples from \cite{KZ}
show that the isoperimetric sets for $\a = 1/2$ in $\R^2$ can be bounded
or the complement of a bounded set.  
We don't have a precise conjecture about how to rule out these examples,
so instead we formulate a question.  
\begin{question} \label{conj:logconcave.cone}   Are there conditions
on log-concave measures  $\mu$ supported
on all of $\R^n$ that include the case $\mu = 1_\O \, dx$ (for bounded
convex $\O$)  as a limit, for which for the following holds?
For every isoperimetric set $E$ with volume fraction $\a$ sufficiently close to $1/2$, 

a) Both $E$ and $E^c = \R^n \setminus E$ are unbounded.

b)  There is a constant $c(\a)>0$ and 
open, convex cones $\G_1\subset E$ and $\G_2 \subset E^c$ such
that  $\mu(\G_i) \ge c(\a)\mu(\R^n)$, $i = 1,\, 2$. 
\end{question}

For reasons explained above, 
related to the behavior of level sets of eigenfunctions, we are most concerned 
with the case of centrally symmetric convex bodies.   
With an eye towards constraints that may ultimately yield 
that the interface is a Lipschitz graph we conjecture that
the the interface is trapped between symmetric cones. 

\begin{conjecture} \label{conj:twohyp.infty}  
There is a constant $c>0$ such that if $\O\subset \R^n$ is
convex, bounded, and centrally symmetric, and $E$ is an isoperimetric
subset of $\O$ with volume ratio $|E|/|\O| = 1/2$, then 

a) there is a convex cone $\G$ such that 
\[
\O \cap \G \subset E, \quad \O \cap (-\G) \subset \O \setminus E,  \quad 
|\G \cap \O| \ge c |\O|.
\]

b) The cone $\G$ can be taken to be a half space.
\end{conjecture}
Conjecture \ref{conj:twohyp.infty} (b), which we call the {\em two hyperplane} conjecture,
says that the interface $\partial E \cap \partial E^c$ is
trapped between two hyperplanes in a very strong dimension-independent
form in the spirit of a famous conjecture of Kannan, Lov\'asz and Simonovits in theoretical computer science known as the KLS Hyperplane Conjecture.   
Their conjecture says that, under the hypotheses of Conjecture \ref{conj:twohyp.infty},
there is an absolute constant $C$ such that 
\[
\min_{a\neq 0}  H_{n-1}(\{x\in \O: a\cdot x= 0\}) \ \le \ C H_{n-1}(\O \cap \partial E).
\]
In other words, the least area hyperplane cut, dividing $\O$ in half by volume, has
area comparable to the least area among all possible ways to
divide $\O$ in half, and  the constant of comparability is independent
of dimension.  
The KLS conjecture implies a long list of famous conjectures in high dimensional convex geometry.   
Both the hyperplane and two-hyperplane conjecture say that least perimeter hypersurfaces resemble hyperplanes.

We shall see below that thanks to an argument of Emanuel Milman, Conjecture
\ref{conj:twohyp.infty} (b) implies the KLS hyperplane conjecture. 
But unlike the hyperplane conjecture, the two-hyperplane conjecture has significance even 
if we allow the constant to depend on dimension.

\begin{question} \label{quest:logconcave.twohyp}  
Consider a log-concave probability measure $\mu = e^{-V(x)}  dx$
with central symmetry.  What additional conditions on $V$ 
guarantee the following?  There is an absolute constant $b^*>0$, such that
for every isoperimetric subset $E$ with $\mu(E) = 1/2$, there is a half space $H$ 
such that 
\[
\mu(H) \ge b^*, \quad  
H \subset E,  \quad  -H \subset E^c = \R^n \setminus E.
\]
\end{question}
A more realistic request may be for a property that guarantees a half space
on one side and a convex cone on the other.

Denote weighted $(n-1)$-Hausdorff measure by $\mu_{n-1} = w H_{n-1}$.  
Here is a precise statement that the two hyperplane
conjecture implies the KLS hyperplane conjecture, with the dependence
of the constants.
\begin{prop} \label{prop:twoimpliesone} (E. Milman \cite{EMilman-personal})  
If $\mu$ is a symmetric, log-concave probability
measure and  $E$ is an isoperimetric set with $\mu(E) = 1/2$
satisfying 
\[
\mu(H) \ge b^*>0, \quad  H \subset E,  \quad  -H \subset E^c = \R^n \setminus E,
\]
and $a$ is the unit normal to the half space $H$, then
\[
P_\mu(E) = \mu_{n-1}(\partial E) \ge \frac{1}{4\log(1/b^*)} \mu_{n-1} (\{x: a \cdot x = 0\}).
\]
\end{prop}

\begin{proof}
The proof, taken from notes by E. Milman, is based on the geometric approach of \cite{EMilmanGeometricApproachPartI, EMilmanIsoperimetricBoundsOnManifolds} to proving that ``concentration implies isoperimetry" in the class of log-concave measures. 
By standard stability results (see e.g. \cite{EMilman-RoleOfConvexity}), it is enough to establish the KLS conjecture when the density $w$ is assumed strictly log-concave 
and smooth on $\R^n$.  

Recall that for a smooth  hypersurface $S \subset \R^n$
with normal $\nu$, the first variation of area of $S$ in the normal direction
is the trace of its second fundamental form at $x \in S$ and 
equals $n-1$ times the usual mean-curvature $H_S(x)$.
Following Gromov, we define the weighted mean-curvature 
$H_{S,\mu}$ as the first variation of $\mu_{n-1}$-weighted area in the normal direction.
It is easy to check  (see (3.1) of \cite{RCBM}) that 
\begin{equation}\label{eq:weightedmc}
H_{S,\mu} = (n-1)H_S - \nu \cdot \nabla V \quad  \mbox{with} \quad w = e^{-V}.
\end{equation}

Results of F. Morgan for weighted area, based on classical results of De Giorgi, Federer, Simons, Almgren, and others (see Appendix A of \cite{EMilman-RoleOfConvexity} and \cite{EMilmanGeometricApproachPartI} and the references therein) then ensure that $E$ 
has a unique open representative modulo 
Lebesgue null-sets, and $\partial E$ is the disjoint union of a closed singular part $\partial_s E$ having Hausdorff dimension at most $n-8$, and a regular part $\partial_r E$ which is a smooth 
$(n-1)$-dimensional manifold; in particular, $\mu_{n-1}(\partial_r E) = \mu_{n-1}(\partial E)$. 
Let $\nu$ denote the outer unit-normal to $E$ on $\partial_r A$. It then easily follows that 
$\partial_r E$ must have \emph{constant} weighted mean-curvature 
\begin{equation}\label{eq:w.cmc}
H_{\partial_r E,\mu}  \equiv H_*, \quad H_* \in \R
\end{equation}

The idea, going back to Gromov \cite{GromovGeneralizationOfLevy}, is to apply the Heintze-Karcher Jacobian comparison theorem to the normal map from the regular part $\partial_r E$. We will require a generalized version of this theorem, due to Morgan (see \cite{EMilmanGeometricApproachPartI}), stating that if $\Pi : S \times \R\rightarrow \R^n$ denotes the normal map $\Pi(x,r) = x + r \nu$, and assuming that $\text{Hess} V \geq K \cdot \text{Id}$, then
\[
\mu(\Pi(S \times [0,t])) \leq \int_S \int_0^t e^{H_{S,\mu}(x) s - K s^2 / 2} \, ds \,  d\mu_{n-1}(x) 
\;\;\; \mbox{for all} \ t > 0 .  
\]
In our setting, $H_{\partial_r E,\mu}$ is constant and $K=0$, and so we see that
\[
\mu(\Pi(\partial_r E \times [0,t])) \leq \mu_{n-1}(\partial E) \int_0^t e^{H_* s}\,  ds \;\;\; 
\mbox{for all} \  t > 0 . 
\]
The crucial observation by Gromov is that any point outside of $E$ will be reached by a normal ray emanating from $\partial_r E$. This follows by observing that
for every $x\in \O$ at distance $t$ from $E$, 
\[
\de B_t(x) \cap \de E  \subset \de_r E.
\]
These points of tangency satisfy an exterior ball condition and hence
are regular by the De Giorgi-Federer theory characterizing
the regular part of the boundary as those points whose tangent cone 
is a hyperplane.   
Consequently, we deduce that
\[
\mu(E_t \setminus E) \leq \mu_{n-1}(\partial E) \int_0^t  e^{H_* s}  \, ds  \;\;\; \mbox{for all} \ 
t > 0 ,
\]
where $E_t$ denotes the points of (Euclidean) distance at most $t$ from $E$. 

We now proceed as in \cite{EMilmanIsoperimetricBoundsOnManifolds}. Since changing the sign of the normal of $\partial_r E$ simply flips the sign of $H_*$, the weighted mean-curvature will be non-positive in at least one of these orientations. Hence, using either 
$E$ or its complement $E^c$ in the reasoning above, we obtain
\[
\min(\mu(E_t \setminus E) , \mu((E^c)_t \setminus E^c)) \leq t \, \mu_{n-1}(\partial E) \;\;\; 
\mbox{for all} \  t > 0 . 
\]
Note that $H \subset E$ and $-H \subset E^c$ implies that $E_L\supset (-H)^c$ and $(E^c)_L \supset H^c$, where $L$ denotes the distance between the two parallel half spaces 
$\pm H$.  By pulling the two half spaces further apart, we may assume 
that $\mu(H) = b^*/2 $ instead of $\mu(H) \ge b^*$.  Keeping
in mind that $b^* \le 1/2$, we get
\[
L \, \mu_{n-1}(\partial E) \ge 
\min(\mu(E_L \setminus E) , \mu((E^c)_L \setminus E^c))  \geq (1/2) - (b^*/2) \ge 1/4.
\]
Therefore, 
\begin{equation} \label{eq:main}
\mu_{n-1}(\partial E) \geq 
\frac{1}{ 4 L } . 
\end{equation}

Recall that $a$ is the unit normal to $H$, and 
denote the marginal by 
\[
w_a(t) = \mu_{n-1}(\{x\in \R^n: a\cdot x = t\}), \quad t\in \R.
\]
The well known theorem of Pr\'ekopa says that $w_a$ is 
log concave.   For all log concave, symmetric weights $w_a$,
\begin{equation} \label{eq:1dim}
\int_t^\infty w_a(s) \, ds \le \frac12 e^{-2w_a(0)t} \quad \mbox{for all} \quad t \ge 0.
\end{equation}
A proof can be found in  \cite{Bob}, but we will carry it out below for completeness.

Applying \eqref{eq:1dim}, we obtain
\[
b^*/2 = \mu(H) = \int_{L/2}^\infty w_a(s)\, ds \le \frac12 e^{-w_a(0)L},
\]
which implies $w_a(0)L \le \log(1/b^*)$.   Finally, applying \eqref{eq:main},
we have 
\[
\mu_{n-1}(\partial E) \ge \frac1{4L} \ge \frac{w_a(0)}{4\log(1/b^*)} = 
\frac{1}{4\log(1/b^*)} \mu_{n-1}(\{x\in \R^n: a\cdot x = 0\}).
\]
This finishes the proof, except for \eqref{eq:1dim}.

\noindent{\em Proof of \eqref{eq:1dim}}.  \ Keeping in
mind the case of equality, $w(s) = (m/2) e^{-m|s|}$, define
$h(s)$  for $s\in \R$ by $w_a(s) = e^{-h(s)}$, set $m = 2w_a(0) = 2e^{-h(0)}$,
and fix $t\ge 0$. 

\noindent
{\em Case 1.} \   $h(t) \ge h(0)  + mt$. 

By convexity of $h$, we have
\[
h(s) \ge h(0) + ms, \quad \mbox{for all} \quad  s \ge t,
\]
and hence,
\[
\int_t^\infty w_a(s) \, ds \le \int_t^\infty e^{-h(0) - ms} \, ds = \frac12 e^{-mt}.
\]

\noindent
{\em Case 2.} \  $ h(t) \le h(0)  + mt$. 

By convexity of $h$, we have 
\[
h(s) \ge h(0) + ms, \quad \mbox{for all} \quad 0 \le s \le  t.
\]
Hence,
\[
\int_0^t w_a(s) \, ds 
\ge \int_0^t e^{-h(0) - ms} \, ds = \frac12\left(1 - e^{-mt}\right),
\]
and
\[
\int_t^\infty  w_a(s) \, ds  = \frac12 - \int_0^t w_a(s) \, ds  \le \frac12 e^{-mt}.
\]
\end{proof}

The two-hyperplane conjecture may be out of reach, given
that it implies the KLS Conjecture.  But confirming or disproving any of these conjectures in low dimensions 
may still shed light on the hot spots problem and related problems
for level sets of elliptic equations.

We consulted with Larry Guth about whether variants of the
foregoing conjectures could be valid in the context of manifolds
with positive Ricci curvature.  He suggested starting with
the question  of whether there are any ``large" convex sets
at all. 

\begin{question}\label{quest:GuthRicci}  (L. Guth)
Is there a dimensional constant $c(n)>0$ such that 
every compact $n$-manifold $M$ with strictly positive Ricci curvature
has a geodesically convex subset $\G$ satisfying
 $\mbox{\rm vol}(\G) \ge c(n) \mbox{\rm vol}(M)$?
\end{question}

\section{Symmetry Breaking}

For centrally symmetric densities and general volume fractions $\a$, 
we have a conjecture that is slightly different from Conjecture \ref{conj:logconcave.hull}.
\begin{conjecture}   \label{conj:logconcave.twocone}
For $0 < \a < 1$, there is $b= b(\a)>0$ such that
for every strictly log-concave probability measure $\mu = w(x) dx$ supported on all of 
$\R^n$ with $w$ centrally symmetric, $w(-x) = w(x)$, and every isoperimetric 
subset $E$ with $\mu(E) = \a$, there is a convex set $\G$ such
that 
\[
\mu(\G) \ge b(\a),  \quad  \G \subset E,  \quad  -\G \subset E^c = \R^n \setminus E.
\]
\end{conjecture}
Conjecture \ref{conj:logconcave.twocone} rules out $E = -E$.
Likewise a positive answer to Question \ref{quest:logconcave.twohyp} rules
out $E = -E$.  Results, as opposed to conjectures, of Sternberg and Zumbrun and
of Rosales et al, establish this symmetry breaking.

\begin{theorem}\label{thm:asymmStZ} (Sternberg and Zumbrun \cite{StZ3})
Suppose that $\O\subset \R^n$ is a smooth, strictly convex domain in
the sense that the second fundamental form of $\de \O$ is
strictly positive on every tangent vector to $\de \O$, and suppose that
$E$ 
is a subset of $\O$ that is stationary for the isoperimetric
problem.  If $-\O = \O$ and $E = -E$, then 
$E$ is unstable. 
\end{theorem}

Sternberg and Zumbrun were interested in particular in the case of the Simons
cone $S = \{(x,y)\in \R^n\times \R^n: |x| = |y|\}$.  Let 
\[
E  = \{(x,y): |x|<|y|, \ |x|^2 + |y|^2<1\} 
\subset B = \{(x,y): |x|^2 + |y|^2<1\}.
\]
Then $B\cap \de E = B\cap S$ and $E$ has mean curvature zero.
Thus it is stationary for the least area and isoperimetric problem in B. 
Moreover, as Simons famously observed,  it is stable
under pertubations that fix the boundary $S\cap \de B$ when $n\ge 4$. 
Nevertheless, by Theorem \ref{thm:asymmStZ}, $E$ is unstable
for the isoperimetric problem.  Indeed, the perturbation that
works is just a translation suitably modified so that the perturbation
divides $B$ exactly in half.   

The Simons cone provides the first potential barrier to proving 
our very first qualitative conjecture Conjecture \ref{conj:hull} (a).
Define
\[
\O : =\mbox{hull}(S\cap B); \quad E_1 := \{(x,y)\in\O: |x|>|y|\}.
\]
Then $S$ divides $\O$ in half and $E_1$ is stationary for the isoperimetric
problem in $\O$.  Theorem \ref{thm:asymmStZ} does not apply
because the boundary of $\O$ is not smooth and not strictly
convex. But the curvature is singular and the points where $S\cap\de \O$,
so in some sense the curvature is very large at the points that matter.
In fact, an explicit calculation with translations shows that
$E_1$ is unstable for the isoperimetric problem in $\O$.  In
other words, the Simons cone does not disprove Conjecture 
\ref{conj:hull} (a).

The variant of the symmetry breaking theorem of Sternberg and Zumbrun for 
log-concave densities is as follows.   

\begin{theorem}\label{thm:asymm}  Let $d\mu = e^{-V} dx$ be a
probability measure that is strictly log-concave and symmetric: $\nabla^2 V >>0$
and $V(-x) = V(x)$.  Suppose that $E$ is a stationary set for the isoperimetric 
problem and $E = -E$.  Then $E$ is unstable.   
\end{theorem}

\noindent
{\bf Proof.}   Note that the argument below was also carried
out by Rosales et al \cite{RCBM} in the special case that $w$ 
is radial and $E$ is a ball in order to rule out balls
as  isoperimetric sets for strictly log-concave radial densities.
The present version relies, in addition, on the identity \eqref{eq:Aidentity}
stated below and used by Sternberg and Zumbrun.

V. Bayle calculated the second variation and hence
the condition for stability. Namely, 
with $w = e^{-V}$, $S = \de E$, and $\nu$ the unit normal to $S$, stability
means 
\[
\int_{S} [|A_S|^2 + \nabla^2 V(\nu ,\nu)]f^2 \, w\, dH^{n-1}
\le \int_{S} |\nabla_S f|^2 \, w\, dH^{n-1}
\]
for all $f\in C^1(S)$ such that 
\[
\int_{S} f \, w \, dH^{n-1} = 0.
\]
Let $e_j$ be an orthonormal basis for $R^n$, 
and define test functions $f_j(x) = e_j\cdot \nu(x)$, $j=1,\ 2,\, \dots, \, n$.  
Then we have 
\[
\int_{S} f_j \, w\, dH^{n-1} = 0, \quad j = 1, \ 2, \ \dots\, , \, n.
\]
because $w(x) = w(-x)$, $-S = S$ and $\nu(-x) = -\nu(x)$.  Furthermore,
a direct calculation (see \cite{StZ3}) yields
\begin{equation}\label{eq:Aidentity}
\sum_{j=1}^n |\nabla_S f_j|^2 = |A_S|^2, \quad \sum_{j=1}^n f_j^2 = 1. 
\end{equation}
If $E$ were stable we could add the stability inequalities for $f_j$ for all $j$
and obtain
\[
\int_S \nabla^2 V(\nu,\nu) \, w \, dH^{n-1} \le 0.
\]
This contradicts $\nabla^2 V >>0$.

Thus the isoperimetric sets in these symmetric settings
break the symmetry. These results can
be viewed as the first qualitative steps pointing toward
an  as yet unrealized quantitative form
of connectivity of isoperimetric sets at the largest scale.
The hope is that such large scale estimates are valid and that,
with the help of scale-invariant connectivity estimates
described earlier, they lead via deformation arguments to 
full regularity in all dimensions of isoperimetric interfaces.

An analogous, qualitative symmetry breaking theorem for
Neumann eigenfunctions was proved by  B. Klartag (see Corollary 1 \cite{Kl}).
This is the first qualitative evidence in favor of our conjectured
monotonicity of the first non-constant Neumann eigenfunction in
symmetric convex domains.

\end{document}